\documentclass[reqno]{amsart}
\usepackage{amsmath,amssymb}
\usepackage{hyperref}
\newcommand*{\mailto}[1]{\href{mailto:#1}{\nolinkurl{#1}}}
\newcommand{\doi}[1]{\href{http://dx.doi.org/#1}{doi:#1}}

\newtheorem{theorem}{Theorem}[section]
\newtheorem{proposition}[theorem]{Proposition}
\newtheorem{definition}[theorem]{Definition}
\newtheorem{lemma}[theorem]{Lemma}
\newtheorem{corollary}[theorem]{Corollary}
\theoremstyle{definition}

\newtheorem{remark}[theorem]{Remark}

\numberwithin{equation}{section}

\begin{document}

\title[Perturbation Theory for Sturm--Liouville Operators]{Relative oscillation theory and essential spectra of Sturm--Liouville operators}

\author[J. Behrndt]{Jussi Behrndt}
\address{Technische Universit\"{a}t Graz\\
Institut f\"ur Angewandte Mathematik\\
Steyrergasse 30\\
8010 Graz\\ Austria}
\email{\mailto{behrndt@tugraz.at}}
\urladdr{\url{https://www.applied.math.tugraz.at/~behrndt/}}

\author[P. Schmitz]{Philipp Schmitz}
\address{Department of Mathematics\\
	Technische Universit\"at Ilmenau\\ Postfach 100565\\
	98648 Ilmenau\\ Germany}
\email{\mailto{philipp.schmitz@tu-ilmenau.de}}
\urladdr{\url{https://www.tu-ilmenau.de/obc/team/philipp-schmitz}}

\author[G. Teschl]{Gerald Teschl}
\address{Faculty of Mathematics\\ University of Vienna\\
Oskar-Morgenstern-Platz 1\\ 1090 Wien\\ Austria}
\email{\mailto{Gerald.Teschl@univie.ac.at}}
\urladdr{\url{https://www.mat.univie.ac.at/~gerald/}}

\author[C. Trunk]{Carsten Trunk}
\address{Department of Mathematics\\
Technische Universit\"at Ilmenau\\ Postfach 100565\\
98648 Ilmenau\\ Germany}
\email{\mailto{carsten.trunk@tu-ilmenau.de}}
\urladdr{\url{https://www.tu-ilmenau.de/funktionalanalysis}}

\thanks{J. Math. Anal. Appl. {\bf 518}, 126673 (2023)}

\begin{abstract}
We develop relative oscillation theory for general Sturm--Liouville differential expressions   of the form
\[
	 \frac{1}{r}\left(-\frac{\mathrm d}{\mathrm dx} p \frac{\mathrm d}{\mathrm dx} + q\right)
\]
and prove
perturbation results and invariance of essential spectra in terms of the real coefficients $p$, $q$, $r$.
The novelty here is that we also allow perturbations of the weight function $r$ in which case the unperturbed
and the perturbed operator act in different Hilbert spaces.
\end{abstract}

\maketitle

\section{Introduction}

The purpose of this paper is to study relative oscillation theory and related perturbation problems for self-adjoint Sturm--Liouville
operators associated with differential expressions of the form
\begin{equation}\label{tautau}
\tau_j=\frac{1}{r_j}\left(-\frac{\mathrm d}{\mathrm dx}p_j\frac{\mathrm d}{\mathrm dx}+q_j\right),\quad j=0,1,
\end{equation}
in the weighted $L^2$-spaces $L^2((a,b);r_j)$, where $-\infty\leq a<b\leq \infty$. As usual, we impose the standard
assumptions that $1/p_j,q_j,r_j \in L^1_{\mathrm{loc}}(a,b)$ are real-valued and $r_j,p_j>0$ a.\,e.
Our main concern in this note is the essential spectrum of self-adjoint realizations associated with
$\tau_j$ and, in particular, conditions on the coefficients which leave the essential spectrum invariant.

It is well known that only the asymptotic behavior of the coefficients near the singular endpoints is relevant for the essential spectrum.
In particular, the essential spectrum is not affected by boundary conditions or the change of the coefficients on any compact subset of $(a,b)$.
Moreover, by imposing an
additional Dirichlet boundary condition at an interior point, the problem can be reduced to two subintervals with one regular and one singular endpoint;
hence it suffices to consider the case that the endpoint $a$ is regular and $b$ is singular.

As mentioned above we are interested in conditions such that two given self-adjoint Sturm--Liouville operators $T_0$ and $T_1$
related to $\tau_0$ and $\tau_1$ in $L^2((a,b);r_0)$ and $L^2((a,b);r_1)$, respectively, have the same essential spectra.
There is a vast literature on this topic for the special case $r_0=r_1$, we mention here only \cite{Weidmann87}, where a good
introduction and further references can be found.

However, the general case $r_0\ne r_1$ has not obtained much attention and to the best of our knowledge there is no (nontrivial) criterion available. From the intuition and our introductory remarks one would expect the essential spectrum to remain unchanged if the
coefficients of $\tau_0$ and $\tau_1$ have the same asymptotic behavior.
 In fact, if
 \begin{equation*}
	 \lim\limits_{x\rightarrow b} \frac{r_1(x)}{r_0(x)}=1,\quad\lim\limits_{x\rightarrow b} \frac{p_1(x)}{p_0(x)}=1,\quad \lim\limits_{x\rightarrow b} \frac{q_1(x)-q_0(x)}{r_0(x)}=0,
\end{equation*}
and $q_0/r_0$ is bounded near $b$, then it turns out in Theorem~\ref{thm:nonoscingap2} that
$\tau_0$ is limit point at $b$ if and only if $\tau_1$ is limit point  at $b$, both operators
$T_0$ and $T_1$ are semibounded from below, and
$$\sigma_\mathrm{ess}(T_0)=\sigma_\mathrm{ess}(T_1).
$$

The key feature in our proof
is relative oscillation theory, which is discussed in Section~\ref{sec2} for general Sturm--Liouville differential expressions of the form \eqref{tautau} along the lines of
\cite{GesztesySimonTeschl96,kt1,kt2,kt3}. Roughly speaking, relative oscillation theory is an
analog of classical oscillation theory for Sturm–Liouville
operators which, rather than measuring the spectrum of one single operator, measures
the difference between the spectra of two different operators.
This is done by replacing zeros of solutions of one operator by weighted zeros of
Wronskians of solutions of two different operators.
Besides the essential spectrum we are also interested in the possible accumulation of eigenvalues to the boundary points of the essential spectrum.
In this context we note that the relative nonoscillatory property in Theorem~\ref{thm:nonoscingap2}~(iv) does not directly apply to boundary points of the essential spectrum and hence
further assumptions on the coefficients are needed to conclude Kneser type results in the spirit of \cite{kt3}; cf. \cite{Kneser93} and also \cite{DunfordSchwartz,GesztesyUnal98,H48,Hi48,RB,Schmidt00,W12}. Here we first formulate Theorem~\ref{thm:gu} as a straightforward generalization of \cite[Theorem 2.1]{kt3}
to obtain sufficient criteria for accumulation and non-accumulation of eigenvalues to the bottom
of the essential spectrum in Theorem~\ref{blaubaer} and Corollary~\ref{cor333}.
These results contain as a special case a variant of Kneser's classical criterion for general Sturm--Liouville operators
of the form ~\eqref{tautau}; cf. Corollary~\ref{kneserli}.

We remark that in the present paper we are only interested in the question whether two given operators are relatively oscillatory or
not. Relative oscillation theory can also be used to compute the precise number of eigenvalues, see \cite{kt1,kt2} (or \cite[Sect.~5.5]{teo}
for a textbook style introduction in the case of regular operators). Relative oscillation theory can also be done in terms of the Maslov
index \cite{HS22}, which is particularly convenient in the case of Sturm--Liouville systems.

\vskip 0.2cm\noindent\\
    {\bf Acknowledgements}.
Jussi Behrndt gratefully acknowledges financial support by the Austrian Science Fund (FWF): P 33568-N.
    	This publication is based upon work from COST Action CA 18232
MAT-DYN-NET, supported by COST (European Cooperation in Science and
Technology), www.cost.eu.

\section{Relative oscillation theory in a nutshell}\label{sec2}

\subsection{Preliminaries}
In this section we recall some results from oscillation theory. An easy introduction in the case of regular problems can be found in \cite{teo}, for more advanced results
we refer to \cite{GesztesySimonTeschl96,Weidmann87,zettl}. Our focus will be on the necessary modifications to accommodate the case $r_0\ne r_1$.

Consider two Sturm--Liouville differential expressions
\begin{equation}
	\label{cheesecake}
	\tau_j=\frac{1}{r_j}\left(-\frac{\mathrm{d}}{\mathrm{d} x}p_j \frac{\mathrm{d} }{\mathrm{d} x} + q_j\right),\quad\text{where }j=0,1,
\end{equation}
on an open interval $(a,b)$ with finite left endpoint $a$ and we shall impose the following conditions
\begin{equation}
	\label{torte}
	\begin{cases}
		p_j,\,q_j,\,r_j \text{ are real-valued functions on }(a,b),\\
		p_j(x)>0,\ r_j(x)>0 \text{ for almost all }x\in (a,b),\\
		1/p_j,\,q_j,\,r_j\in L^1_\mathrm{loc}(a,b),\\
		\tau_j\text{ is regular at }a
	\end{cases}
\end{equation}
for $j=0,1$. Note that since we are interested in the essential spectra of self-adjoint realizations of $\tau_j$, the assumption that $a$ is regular can be made without loss of
generality.

Recall that a nontrivial real-valued solution $u_j$ of $(\tau_j-\lambda)u=0$, $\lambda\in\mathbb R$, can be represented in terms of Prüfer variables, that is, there are absolutely continuous functions $\rho_{u_j}$ and $\theta_{u_j}$ such that
\begin{equation}\label{pruf}
u_j(x)=\rho_{u_j}(x)\sin(\theta_{u_j}(x))\quad\text{and} \quad (p_ju_j')(x)=\rho_{u_j}(x)
\cos(\theta_{u_j}(x)),
\end{equation}
where the Prüfer radius $\rho_{u_j}$ is positive and the Prüfer angle $\theta_{u_j}$ is uniquely
determined once a value of $\theta_{u_j}(x_0)$ is chosen by requiring
continuity of $\theta_{u_j}$. It  satisfies the differential equation
\begin{equation}
	\label{prueferequ}
	\theta_{u_j}' = \frac{1}{p_j} (\cos\theta_{u_j})^2 - (q_j-\lambda r_j) (\sin \theta_{u_j})^2.
\end{equation}
One verifies that the Prüfer angle is strictly increasing at the zeros of the solution $u_j$ and it follows that
the number of zeros of $u_j$ in $(a,x)$ is given by
\begin{equation}
		\label{Nu}
		N_{u_j}(x) := \left\lceil\frac{\theta_{u_j}(x)}{\pi}\right\rceil - \left\lfloor\frac{\theta_{u_j}(a)}{\pi}\right\rfloor - 1,\quad x\in(a,b),
	\end{equation}
where $\lceil\cdot\rceil$ is the ceiling function  and $\lfloor\cdot\rfloor$ the floor function.
For every $x\in (a,b)$ the solution $u_j$ has at most finitely many zeros in $(a,x)$. We note that
the function $N_{u_j}:(a,b)\rightarrow \mathbb{Z}$ is non-negative and increasing.

In the following let $\lambda\in\mathbb R$ and recall that $\tau_0-\lambda$ is said to be \emph{nonoscillatory} if there is a nontrivial real-valued solution
$u$ of $(\tau_0-\lambda) u=0$ with at most finitely many zeros in $(a,b)$, that is, $\lim_{x\rightarrow b} N_u(x)<\infty$.
Otherwise, $\tau_0-\lambda$ is called \emph{oscillatory}. We note that this property is independent of the choice of the solution.
The  number of zeros of a solution of $(\tau_0-\lambda)u=0$ is closely related to the spectra of the self-adjoint realisations of $\tau_0$. More precisely, if $T_0$
is some self-adjoint realisation of $\tau_0$ in
the weighted Hilbert space $L^2((a,b);r_0)$  and $E_0(\cdot)$ denotes the spectral measure of $T_0$ then
\begin{equation}\label{zuj}
\dim\operatorname{ran} \bigl(E_0((-\infty,\lambda))\bigr)<\infty\quad\text{if and only if}\quad \lim_{x\rightarrow b} N_{u}(x)<\infty
\end{equation}
for some (and hence for all)  nontrivial real-valued solutions $u$ of $(\tau_0-\lambda) u=0$. Furthermore, if $-\infty<\lambda<\mu<\infty$
and $u$ and $v$ are nontrivial real-valued solutions of $(\tau_0-\lambda)u=0$ and $(\tau_0-\mu)v=0$, respectively, then
\begin{equation}\label{zuj2}
\dim\operatorname{ran} \bigl(E_0((\lambda,\mu))\bigr)<\infty \quad\text{if and only if}\quad \liminf_{x\rightarrow b} \bigl(N_v(x) - N_u(x)\bigr)<\infty.
\end{equation}
Note that by \eqref{zuj} $T_0$ is semi-bounded from below if and only if there is $\lambda\in\mathbb R$ such that $\lim_{x\rightarrow b} N_{u}(x)<\infty$, that is,
$\tau_0-\lambda$ is nonoscillatory. In this case $\tau_0-\lambda$ is nonoscillatory for all $\lambda<\inf\sigma_{\mathrm{ess}}(T_0)$.
%, and $\tau_0-\lambda$ is oscillatory at $\lambda = \inf\sigma_{\mathrm{ess}}(T_0)$ if and only if the set $\sigma(T_0)\cap (-\infty,\lambda)$ consists of an infinite sequence of isolated eigenvalues of $T_0$ which converge to $\lambda$.

\subsection{Relative oscillation theory}
The central object in this section is the \emph{modified Wronskian} and its zeros.
For solutions $u_0$ and $u_1$ of two different Sturm--Liouville differential expressions,
$$
 (\tau_0-\lambda_0) u_0=0
 \qquad \mbox{and} \quad (\tau_1-\lambda_1)u_1=0,
$$
at two different real values $\lambda_0,\lambda_1$ the \emph{modified Wronskian} is defined by
\begin{equation*}
	%\label{soy2}
	W(u_0,u_1)(x):= u_0(x)\, (p_1u_1')(x) - (p_0u_0')(x)\, u_1(x),\qquad x\in(a,b).
\end{equation*}
%Observe that $W(u_0,u_1)(x)=0$ if and only if the $\mathbb{C}^2$-vectors $(u_0(x),(p_0u_0')(x))^\top$ and $(u_1(x),(p_1u_1')(x))^\top$ are linearly dependent.
In the case of real-valued nontrivial solutions  $u_0$ and $u_1$  one obtains from \eqref{pruf}
\begin{equation*}
	%\label{calling}
	W(u_0,u_1)(x) = \rho_{u_0}(x) \rho_{u_1}(x) \sin\bigl( \theta_{u_0}(x)-\theta_{u_1}(x)\bigr)
\end{equation*}
and hence $W(u_0,u_1)(x)=0$ if and only if $\theta_{u_1}(x)-\theta_{u_0}(x) = k\pi$
for some $k\in \mathbb{Z}$. We consider the function
\begin{equation}
\label{test}
N(u_0,u_1)(x) := \left\lceil\frac{\theta_{u_1}(x)-\theta_{u_0}(x)}{\pi}\right\rceil -
\left\lfloor\frac{\theta_{u_1}(a)-\theta_{u_0}(a)}{\pi}\right\rfloor -1,\quad x\in(a,b).
\end{equation}

\begin{remark}
Nontrivial solutions (when considered as vector-valued solutions $(u,pu')$ of the associated system) correspond to a path of one-dimensional
Lagrangian subspaces and hence these subspaces can be identified with the corresponding Pr\"ufer angles. In particular, two such path
cross whenever the Pr\"ufer angles agree modulo $\pi$ and hence whenever the Wronskian of the two solutions vanishes.
Consequently, \eqref{test} can be identified with the Maslov index of the two solutions on the interval $(a,x)$ (cf.\ \cite{HS22}).
\end{remark}

Let $u_2$ be a real-valued nontrivial solution of  $(\tau_2-\lambda_2)u=0$, where  $\tau_2$ is a differential
expression of the form \eqref{cheesecake} satisfying \eqref{torte}.
It follows from \eqref{Nu} and the properties of the  ceiling function $\lceil\cdot\rceil$ and the floor function $\lfloor\cdot\rfloor$ that
\begin{equation}
		\label{lady}
		N_{u_1}(x) - N_{u_0}(x) -3 \leq N(u_0,u_1)(x)  \leq N_{u_1}(x) - N_{u_0}(x) + 1,
	\end{equation}
	\begin{equation}
		\label{klotz}
		-N(u_1,u_0)(x)-2\leq N(u_0,u_1)(x)\leq -N(u_1,u_0)(x),\mbox{ and}
	\end{equation}
	\begin{equation}
		\label{gaga}
		N(u_0,u_1)(x) + N(u_1,u_2)(x) -1 \leq N(u_0,u_2)(x) \leq N(u_0,u_1)(x) + N(u_1,u_2)(x) +1
	\end{equation}
	for all $x\in (a,b)$.

\begin{lemma}
	\label{gabel}
	Suppose that \eqref{torte} holds for $j=0$. Let $u$ and $v$ be nontrivial real-valued solutions of $(\tau_0-\lambda)u=0$ for $\lambda\in\mathbb R$. If $u$ and $v$ are linearly dependent solutions then $N(u,v)(x)=-1$ for all $x\in (a,b)$. Otherwise $N(u,v)(x)=0$ for all $x\in (a,b)$.
\end{lemma}
\begin{proof}
	Since $u$ and $v$ are solutions of the same differential equation, the Wronskian is constant on $[a,b)$. If $u$ and $v$ are linearly dependent then the Wronskian vanishes everywhere and due to the representation by means of Prüfer variables we see $\theta_v(x)-\theta_u(x)=k\pi$ for all $x\in [a,b)$ and a suitable $k\in\mathbb{Z}$. This implies $N(u,v)(x)=-1$ for all $x\in (a,b)$. Otherwise, if both functions are linearly independent then the Wronskian has no zeros in $[a,b)$. Hence, the difference of Prüfer angles $\theta_v-\theta_u$ does not attain any integer multiple of $\pi$. By continuity we have $\theta_v(x)-\theta_u(x)\in (k\pi,(k+1)\pi)$ for all $x\in [a,b)$ and some $k\in\mathbb{Z}$, which shows $N(u,v)(x)=0$.
\end{proof}

Under some additional assumptions on the coefficients of $\tau_j$ it turns out that the function $N(u_0,u_1)$ in \eqref{test} has similar properties
	as the functions $N_{u_j}$ in \eqref{Nu}.

\begin{lemma}\label{lemmayy}
Let $u_j$ be real-valued nontrivial solutions of $(\tau_j-\lambda_j)u=0$ for $j=0,1$, and $\lambda_j\in\mathbb R$.
\begin{itemize}
 \item [{\rm (i)}] Assume that the conditions
 \begin{equation}\label{cond1}
  p_0\geq p_1 \quad\text{and}\quad q_0-\lambda_0 r_0 \geq q_1-\lambda_1 r_1
 \end{equation}
 hold. Then $N(u_0,u_1)$ is an increasing function with $N(u_0,u_1)(x)\geq -1$ for all $x\in (a,b)$.
 \item [{\rm (ii)}] Assume that the conditions
 \begin{equation}\label{cond2}
  p_0\geq p_1 \quad\text{and}\quad q_0-\lambda_0 r_0 > q_1-\lambda_1 r_1
 \end{equation}
 hold. Then
 for every $x\in (a,b)$ the Wronskian $W(u_0,u_1)$ has at most finitely many zeros in $(a,x)$ and the value $N(u_0,u_1)(x)$ coincides with the number of zeros of $W(u_0,u_1)$ in $(a,x)$.
\end{itemize}
\end{lemma}

\begin{proof} (i) \
 Let $a\leq \xi<x<b$ and assume that $\theta_{u_1}(\xi)-\theta_{u_0}(\xi)\in [k\pi,(k+1)\pi)$ for some $k\in\mathbb{Z}$.
By \eqref{prueferequ} and the  angle addition formulae
$\sin(\alpha+\beta)\sin(\alpha-\beta)=\cos^2 \beta-\cos^2 \alpha=\sin^2 \alpha-\sin^2 \beta$  we obtain
	\begin{equation*}
		\begin{split}
			\theta_{u_1}'-\theta_{u_0}'
			={ }&\left(\frac{1}{p_1}-\frac{1}{p_0}\right) \cos^2\theta_{u_1} + \bigl((q_0-\lambda_0 r_0)-(q_1-\lambda_1 r_1)\bigr) \sin^2\theta_{u_0}\\[0.5\baselineskip]
			&-(q_1-\lambda_1r_1)\Bigl(\sin^2\theta_{u_1} - \sin^2\theta_{u_0}\Bigr) - \frac{1}{p_0}\Bigl(\cos^2\theta_{u_0}-\cos^2\theta_{u_1}\Bigr)
\\[0.5\baselineskip]
			 ={ }& \left(\frac{1}{p_1}-\frac{1}{p_0}\right) \cos^2\theta_{u_1} + \bigl((q_0-\lambda_0 r_0)-(q_1-\lambda_1 r_1)\bigr) \sin^2\theta_{u_0}\\[0.5\baselineskip]
			&- (-1)^{k}\left(\frac{1}{p_0} + q_1-\lambda_1 r_1\right) \sin(\theta_{u_0}+\theta_{u_1})\sin\delta,
		\end{split}
	\end{equation*}
where $\delta$ stands for $\theta_{u_1}-\theta_{u_0}-k\pi$. We consider the functions
	\begin{equation}
		\label{exile}
		f= \left(\frac{1}{p_1}-\frac{1}{p_0}\right) \cos^2\theta_{u_1} + \bigl((q_0-\lambda_0 r_0)-(q_1-\lambda_1 r_1)\bigr) \sin^2\theta_{u_0}
	\end{equation}
	and
	\begin{equation*}
		h = - (-1)^{k}\left(\frac{1}{p_0} + q_1-\lambda_1 r_1\right) \sin(\theta_{u_0}+\theta_{u_1})\frac{\sin\delta}{\delta}.
	\end{equation*}
	Clearly, we have $\delta' = f + h \delta$,
	where the functions $f$, $h$ are integrable on $(a,c)$ for all $c\in(a,b)$.
	Consider the positive function $g$ given by
	\begin{equation*}
		g(x) = \exp\left(-\int_a^x h(t)\,\mathrm dt\right).
	\end{equation*}
	Then
	\begin{equation}
		\label{567}
		(g\delta)' = - \delta hg + (f+h\delta)g = fg\geq 0
	\end{equation}
	by \eqref{exile} and \eqref{cond1}. Hence, $g\delta$ is an increasing function. %For $x<\xi$ one has
%	\begin{equation}
%		\label{jabbaI}
%		g(x)\bigl(\theta_{u_1}(x)-\theta_{u_0}(x)-k\pi\bigr) = (g\delta)(x) \leq (g\delta)(\xi) = g(\xi)\bigl(\theta_{u_1}(\xi)-\theta_{u_0}(\xi)-k\pi\bigr)
%	\end{equation}
	For $x>\xi$ the estimate
	\begin{equation}
		\label{jabbaII}
		g(x)\bigl(\theta_{u_1}(x)-\theta_{u_0}(x)-k\pi\bigr) = (g\delta)(x) \geq (g\delta)(\xi) = g(\xi)\bigl(\theta_{u_1}(\xi)-\theta_{u_0}(\xi)-k\pi\bigr)
	\end{equation}
	holds. As $\theta_{u_1}(\xi)-\theta_{u_0}(\xi)\in [k\pi,(k+1)\pi)$,
\eqref{jabbaII} implies $\theta_{u_1}(x)-\theta_{u_0}(x)\geq k\pi$ and
	\begin{equation*}
		\left\lfloor\frac{\theta_{u_1}(\xi)-\theta_{u_0}(\xi)}{\pi}\right\rfloor
\leq\left\lceil\frac{\theta_{u_1}(\xi)-\theta_{u_0}(\xi)}{\pi}\right\rceil\leq \left\lceil\frac{\theta_{u_1}(x)-\theta_{u_0}(x)}{\pi}\right\rceil.
	\end{equation*}
This shows $N(u_0,u_1)(\xi)\leq N(u_0,u_1)(x)$ and with $\xi=a$ one sees $N(u_0,u_1)(x)\geq -1$.

(ii) Under the stronger condition \eqref{cond2}, the inequality in
\eqref{567} is strict (almost everywhere in a neighbourhood of $\xi$) and, hence, also the inequality in \eqref{jabbaII}. In particular, we see that for $x>\xi$
\begin{equation}\label{lolo}
 \theta_{u_1}(\xi)-\theta_{u_0}(\xi)\geq k\pi\quad\text{implies}\quad \theta_{u_1}(x)-\theta_{u_0}(x) > k\pi
\end{equation}
and for $x<\xi$ the inequality in  \eqref{jabbaII} changes accordingly and
\begin{equation}\label{lolo2}
 \theta_{u_1}(\xi)-\theta_{u_0}(\xi)\leq k\pi\quad\text{implies}\quad \theta_{u_1}(x)-\theta_{u_0}(x) < k\pi.
\end{equation}
In what follows, choose $x\in (a,b)$ and $k\in \mathbb{Z}$ with
$\theta_{u_1}(a)-\theta_{u_0}(a)\in [k\pi,(k+1)\pi)$ which means
$\left\lfloor \theta_{u_1}(a)-\theta_{u_0}(a)\right\rfloor = k\pi$.
Moreover, by  \eqref{lolo}, we have
$$
\theta_{u_1}(y)-\theta_{u_0}(y) \in (k\pi,\infty) \quad\text{for all}\quad y\in (a,x).
$$

If $\theta_{u_1}(x)-\theta_{u_0}(x) \in (k\pi,(k+1)\pi]$, then
$N(u_0,u_1)(x)=0$ by definition.  By \eqref{lolo} there is no $y \in  (a,x)$
with $\theta_{u_1}(y)-\theta_{u_0}(y) \geq (k+1)\pi$.
Therefore
$$
\theta_{u_1}(y)-\theta_{u_0}(y) \in (k\pi,(k+1)\pi) \quad\text{for all}\quad y\in (a,x).
$$
As the Wronskian $W(u_0,u_1)$ is zero
if and only if $\theta_{u_1}(y)-\theta_{u_0}(y)$ equals $l\pi$  for some
$l\in \mathbb{Z}$, we see that on the interval $(a,x)$ there are no zeros
of the Wronskian.
This coincides with the value of $N(u_0,u_1)(x)$.

If $\theta_{u_1}(x)-\theta_{u_0}(x) \in ((k+1)\pi,(k+2)\pi]$, then
$N(u_0,u_1)(x)=1$ by definition.  By \eqref{lolo} there is no $y \in  (a,x)$
with $\theta_{u_1}(y)-\theta_{u_0}(y) \geq (k+2)\pi$.
Therefore
$$
\theta_{u_1}(y)-\theta_{u_0}(y) \in (k\pi,(k+2)\pi) \quad\text{for all}\quad y\in (a,x).
$$
As the function $\theta_{u_1}-\theta_{u_0}$ is continuous and takes in
$a$ a value below $(k+1)\pi$ and in $x$ a value above $(k+1)\pi$, there
exists $y_1 \in (a,x)$ with $\theta_{u_1}(y_1)-\theta_{u_0}(y_1) =(k+1)\pi$,
which is a zero of the Wronskian. An application of \eqref{lolo} and \eqref{lolo2}
with $\xi=y_1$ shows that this is the only zero of the Wronskian in the interval
$(a,x)$, which coincides with the value of $N(u_0,u_1)(x)$.

If $\theta_{u_1}(x)-\theta_{u_0}(x) \in ((k+2)\pi,(k+3)\pi]$, then
$N(u_0,u_1)(x)=2$ by definition. Similar as above,
 by \eqref{lolo}, there is no $y \in  (a,x)$
with $\theta_{u_1}(y)-\theta_{u_0}(y) \geq (k+3)\pi$ and we conclude with
\eqref{lolo} and \eqref{lolo2} that the Wronskian on the interval
$(a,x)$ has $N(u_0,u_1)(x)=2$ zeros.
Continuing in this way shows the statement.
\end{proof}

An important special case in Lemma~\ref{lemmayy}~(ii)
is the case that $u_0$ and $v_0$ are real-valued solutions of $(\tau_0-\lambda)u=0$ and
	$(\tau_0-\mu)v=0$, respectively, where $\lambda<\mu$. In this situation \eqref{cond2} holds with $p_0=p_1$ and $q_0-\lambda r_0 > q_0-\mu r_0$ and hence
	$N(u_0,v_0)(x)<\infty$ is the number of zeros of the Wronskian $W(u_0,v_0)$ in $(a,x)$.

As a consequence we also conclude the following useful version of Sturm's comparison theorem.
	
	\begin{corollary}[Sturm's comparison theorem]
	\label{Scomparison}
	Let $u_j$ be real-valued nontrivial solutions of $(\tau_j-\lambda_j)u=0$ for $j=0,1$, and $\lambda_j\in\mathbb R$, and let $x_0$ and $x_1$ be consecutive zeros of $u_0$ in $(a,b)$. If the condition \eqref{cond2} holds, then there is at least one zero $y\in (x_0,x_1)$ of $u_1$.
\end{corollary}

\begin{proof}
Let $\theta_{u_0}(x_0)=k\pi$, $\theta_{u_0}(x_1)=(k+1)\pi$, and $\theta_{u_1}(x_0)\in [j\pi,(j+1)\pi)$ for some $k$, $j\in\mathbb{Z}$. Then
	\begin{equation*}
		%\label{deltaDiff}
		(j-k)\pi\leq\theta_{u_1}(x_0)-\theta_{u_0}(x_0)
	\end{equation*}
	and by \eqref{lolo}
	\begin{equation*}
		%\label{deltaDiffII}
		(j-k)\pi<\theta_{u_1}(x_1)-\theta_{u_0}(x_1) = \theta_{u_1}(x_1)-(k+1)\pi.
	\end{equation*}
	Therefore, $\theta_{u_1}(x_1)>(j+1)\pi$ which yields the existence of $y\in (x_0,x_1)$ with $\theta_{u_1}(y)=(j+1)\pi$, that is $u_1(y)=0$.
	\end{proof}

We next introduce the concept of relative oscillation. The following definition is due to Krüger and Teschl \cite{kt1,kt2,kt3}.

\begin{definition}\label{def:relosc}
	For $j=0,1$ and $\lambda_j\in\mathbb R$ consider nontrivial real-valued solutions $u_j$ of $(\tau_j-\lambda_j)u=0$.
	We say that $\tau_0-\lambda_0$ is \emph{relatively nonoscillatory} with respect to $\tau_1-\lambda_1$ if both limits
	\begin{equation*}
		\underline{N}(u_0,u_1) := \liminf_{x\rightarrow b}
		N(u_0,u_1)(x) \quad\mbox{and}\quad \overline{N}(u_0,u_1) :=
		\limsup_{x\rightarrow b} N(u_0,u_1)(x)
	\end{equation*}
	are finite. Otherwise, $\tau_0-\lambda_0$ is called \emph{relatively oscillatory} with respect to $\tau_1-\lambda_1$.
\end{definition}

It turns out that the definition of relative (non)oscillation does not depend on the particular solutions.
In fact, for another pair of nontrivial real-valued solutions $v_0$, $v_1$ of of $(\tau_0-\lambda_0)u=0$ and $(\tau_1-\lambda_1)u=0$, respectively, the inequality \eqref{gaga} applied twice together with Lemma~\ref{gabel} implies
\begin{equation*}
	\begin{split}
		N(v_0,v_1)(x) &\leq N(v_0,u_0)(x) + N(u_0,v_1)(x) +1\\[0.5\baselineskip]
		&\leq N(v_0,u_0)(x) + N(u_0,u_1)(x) + N(u_1,v_1)(x) +2 \leq N(u_0,u_1)(x) +2
	\end{split}
\end{equation*}
and
\begin{equation*}
	\begin{split}
		N(v_0,v_1)(x) &\geq N(v_0,u_0)(x) + N(u_0,v_1)(x) -1\\[0.5\baselineskip]
		&\geq N(v_0,u_0)(x) + N(u_0,u_1)(x) + N(u_1,v_1)(x) -2 \geq N(u_0,u_1)(x) -4
	\end{split}
\end{equation*}
for all $x\in (a,b)$. Hence, the limits $\overline N(u_0,u_1)$ and $\underline N(u_0,u_1)$ are finite if and only if $\overline N(v_0,v_1)$ and $\underline N(v_0,v_1)$ are finite. Furthermore, the notion {\it relatively nonoscillatory} gives rise to an equivalence relation. Below we will use the following facts which are direct consequences of
 \eqref{klotz} and  \eqref{gaga}. For this let $\tau_2$ be a differential
expression of the form \eqref{cheesecake} satisfying \eqref{torte}.
 \begin{itemize}
 \item[{\rm (a)}] If $\tau_0-\lambda_0$ is relatively oscillatory with respect to $\tau_1-\lambda_1$, then
  $\tau_1-\lambda_1$ is relatively oscillatory with respect to $\tau_0-\lambda_0$.
  \item[{\rm (b)}] If  $\tau_0-\lambda_0$ is relatively oscillatory with respect to $\tau_1-\lambda_1$ and
   $\tau_1-\lambda_1$ is relatively oscillatory with respect to $\tau_2-\lambda_2$, then
    $\tau_0-\lambda_0$ is relatively oscillatory with respect to $\tau_2-\lambda_2$.
 \end{itemize}

%as the zeros of two linearly independent real-valued nontrivial solutions  $u_0$ and $u_1$  of $(\tau_j-\lambda_j)u=0$ interlace. Indeed, let  $x_0$ and $x_1$ be consecutive zeros of $u_0$ in $(a,b)$, then the Wronskian $W[u_0,u_1]$ has no zero at $x_0$ and the inequality in \eqref{deltaDiff} is strict. Hence, by Lemma~\ref{lemmayy}(i), one obtains again \eqref{deltaDiffII} and the existence of a zero $y\in (x_0,x_1)$ of $u_1$ follows by the same arguments used in the proof of Corollary \ref{Scomparison}, see also \cite{kt2}.

Note also that under assumption \eqref{cond1}
or \eqref{cond2}
the function $N(u_0,u_1)$ is increasing and hence in that case
\begin{equation}\label{kkk}
 \lim_{x\rightarrow b} N(u_0,u_1)(x)= \underline{N}(u_0,u_1)=\overline{N}(u_0,u_1)\leq \infty.
\end{equation}

The next lemma describes the relationship between classical and relative oscillation; cf. \cite[Lemma~4.5]{kt2}.

\begin{lemma}\label{MyLittlePony}
	Suppose that $\tau_0-\lambda_0$ is nonoscillatory. Then $\tau_1-\lambda_1$ is relatively nonoscillatory with respect to $\tau_0-\lambda_0$ if and only if $\tau_1-\lambda_1$ is nonoscillatory.
\end{lemma}

\begin{proof}
	Let $u_0$ and $u_1$ be nontrivial real-valued solutions of $(\tau_0-\lambda_0)u_0=0$ and $(\tau_1-\lambda_1)u_1=0$, respectively.
	Since $\tau_0-\lambda_0$ is nonoscillatory the solution $u_0$ has at most finitely many zeros in $(a,b)$ and hence we have $0\leq N_{u_0}(x)\leq n_0$ for some
	$n_0\in\mathbb{N}$ and all $x\in (a,b)$. Therefore \eqref{lady} implies
	\begin{equation*}
	\begin{split}
		N_{u_1}(x) - n_0 -3 &\leq N_{u_1}(x) - N_{u_0}(x) - 3\\
		                    &\leq N(u_0,u_1)(x)\\
		                    &\leq N_{u_1}(x) - N_{u_0}(x) + 1 \leq N_{u_1}(x) + 1
    \end{split}
	\end{equation*}
	for all $x\in (a,b)$. This shows that $\lim_{x\rightarrow b} N_{u_1}(x)$ is finite if and only if $(\tau_0-\lambda_0)$ is
	relatively nonoscillatory with respect to $(\tau_1-\lambda_1)$.
\end{proof}

Along the lines of \eqref{zuj2} we obtain a result on the finiteness and infiniteness of the spectrum. Again
$E_0(\cdot)$ denotes the spectral measure of $T_0$; cf.\ \cite[Sect.~14]{Weidmann87}.

\begin{lemma}\label{111}
 Let $T_0$ be a self-adjoint realisation of $\tau_0$ in $L^2((a,b);r_0)$ and fix $\lambda,\mu\in\mathbb R$ with $\lambda<\mu$. Then
 $\dim \operatorname{ran}(E_0((\lambda,\mu)))<\infty$ if and only if $\tau_0-\lambda$ is relatively nonoscillatory with respect to
 $\tau_0-\mu$.
\end{lemma}

\begin{proof}
 Let $u_0$ and $v_0$ be nontrivial real-valued solutions of $(\tau_0-\lambda)u=0$ and $(\tau_0-\mu)v=0$, respectively. Then \eqref{lady}
 and \eqref{kkk} give
	\begin{equation*}
		\liminf_{x\rightarrow \infty}\bigl(N_{v_0}(x)-N_{u_0}(x)\bigr)-3\leq \underline{N}(u_0,v_0) = \overline{N}(u_0,v_0)\leq
		\liminf_{x\rightarrow\infty} \bigl(N_{v_0}(x)-N_{u_0}(x)\bigr)+1
	\end{equation*}
	and hence the statement follows from \eqref{zuj2}.
\end{proof}

Observe that $\dim \operatorname{ran}(E_0((\lambda,\mu)))<\infty$ implies $\dim \operatorname{ran}(E_0((\lambda,\eta)))<\infty$ for all
$\eta\in [\lambda,\mu]$ and hence Lemma~\ref{111} also shows that $\tau_0-\lambda$ is relatively nonoscillatory with respect to
 $\tau_0-\eta$ for all $\eta\in [\lambda,\mu]$.

The next result extends
 \cite[Theorem~4.6]{kt2} to $r_0\neq r_1$.

\begin{theorem}
	\label{KruegerTeschl}
	Let $T_j$ be self-adjoint realizations of $\tau_j$ in $L^2((a,b);r_j)$ with spectral measures
$E_j(\cdot)$ for $j=0,1$. Fix $\lambda,\mu\in\mathbb R$ with $\lambda<\mu$
	and assume that $\dim\operatorname{ran} E_0((\lambda,\mu))<\infty$. If $\tau_0-\lambda$
	is relatively nonoscillatory with respect to $\tau_1-\lambda$ and  $\tau_0-\mu$
	is relatively nonoscillatory with respect to $\tau_1-\mu$, then
	$\dim\operatorname{ran}(E_1((\lambda,\mu)))<\infty$.
	%and that $\tau_0-\eta$ is relatively nonoscillatory with respect to $\tau_1-\eta$
	%for some $\eta\in [\lambda,\mu]$. Then $\dim\operatorname{ran}(P_{T_1}((\lambda,\mu)))<\infty$ if and only if $\tau_0-\nu$
	%is relatively nonoscillatory with respect to $\tau_1-\nu$ for all $\nu\in [\lambda,\mu]$.
\end{theorem}

\begin{proof}
    From $\dim\operatorname{ran} E_0((\lambda,\mu))<\infty$ and Lemma~\ref{111} it follows that $\tau_0-\lambda$ is relatively nonoscillatory with respect to
 $\tau_0-\mu$. As $\tau_0-\lambda$
	is relatively nonoscillatory with respect to $\tau_1-\lambda$ and  $\tau_0-\mu$
	is relatively nonoscillatory with respect to $\tau_1-\mu$ by assumption we conclude
with the properties (a) and (b) from above that $\tau_1-\lambda$ is relatively nonoscillatory with respect to
 $\tau_1-\mu$. With the help of Lemma~\ref{111} we now obtain $\dim \operatorname{ran} E_1((\lambda,\mu))<\infty$.

%     $\dim\operatorname{ran} P_{T_0}([\lambda,\mu])<\infty$ and hence
% 	for $\eta\in [\lambda,\mu]$ we have $\dim\operatorname{ran} P_{T_0}((\lambda,\eta))<\infty$ and hence $\tau_0-\lambda$ is relatively nonoscillatory with respect
% 	to $\tau_0-\eta$ by Lemma~\ref{111}.  If $\dim \operatorname{ran} P_{T_1}((\lambda,\mu))<\infty$, then
% 	the same argument shows that $\tau_1-\lambda$ is relatively nonoscilatory with respect to $\tau_1-\nu$ for $\nu\in [\lambda,\mu]$. We have
% 	\begin{equation*}
% 		(\tau_0-\mu)\sim (\tau_0-\lambda_0)\sim (\tau_0-\lambda)\sim (\tau_1-\lambda) \sim (\tau_1-\lambda_0)\sim (\tau_1-\mu).
% 	\end{equation*}
% 	On the other hand if $(\tau_0-\mu)\sim (\tau_1-\mu)$ for all $\mu\in [\lambda_0,\lambda_1]$ we obtain by transitivity and part (i)
% 	\begin{equation*}
% 		(\tau_1-\lambda_0)\sim (\tau_0-\lambda_0)\sim (\tau_0-\lambda_1)\sim (\tau_1-\lambda_1).
% 	\end{equation*}
% 	By applying (i) once again we see $\dim\operatorname{ran} P_{T_1}((\lambda_0,\lambda_1))<\infty$.
\end{proof}

\section{Essential spectra of Sturm--Liouville operators}

In this section we shall consider the Sturm--Liouville expressions $\tau_j$, $j=0,1$, in \eqref{cheesecake}--\eqref{torte}.
Our main objective is to prove a result on the invariance of the essential spectrum for the self-adjoint realizations of $\tau_j$ in $L^2((a,b);r_j)$.
In this context it seems natural to impose a limit point assumption for the right endpoint $b$; cf. Theorem~\ref{thm:nonoscingap2}~(i).
We start with a useful consequence of Theorem~\ref{KruegerTeschl}, which provides the inclusion of the essential spectra.

\begin{proposition}\label{cor:sigess}
Let $T_j$ be self-adjoint realizations of $\tau_j$ in $L^2((a,b);r_j)$ for $j=0,1$, and assume that $\tau_1-\lambda$ is relatively nonoscillatory with respect to
$\tau_0-\lambda$ for every $\lambda\in\mathbb R\setminus\sigma_\mathrm{ess}(T_0)$. Then $\sigma_\mathrm{ess}(T_1) \subset \sigma_\mathrm{ess}(T_0)$.
\end{proposition}

\begin{proof}
	For $\eta\in \mathbb R\setminus\sigma_{\mathrm{ess}}(T_0)$
	choose $\lambda<\eta<\mu$ such that $[\lambda,\mu]\subset\mathbb R\setminus\sigma_{\mathrm{ess}}(T_0)$. Then we have
	$\dim\operatorname{ran} E_0((\lambda,\mu))<\infty$ and it follows from the assumption that
	$\tau_0-\lambda$ is relatively nonoscillatory with respect to $\tau_1-\lambda$ and $\tau_0-\mu$ is relatively nonoscillatory with respect to $\tau_1-\mu$.
	Now Theorem~\ref{KruegerTeschl} implies
	$\dim\operatorname{ran} E_1((\lambda,\mu))<\infty$ which leads to $\eta\in\mathbb R\setminus \sigma_{\mathrm{ess}}(T_1)$.
\end{proof}

\noindent Next, we obtain a criterion for two Sturm--Liouville differential expressions being relatively nonoscillatory with respect to each other involving all coefficients. The special case $r_0=r_1$ was treated in \cite{kt2}.

\begin{theorem}\label{thm:nonoscingap2}
Let $T_j$ be self-adjoint realizations of $\tau_j$ in $L^2((a,b);r_j)$ for $j=0,1$, and assume the following conditions at the endpoint $b$:
\begin{itemize}
	\item[($\alpha$)] $\lim\limits_{x\rightarrow b} \frac{r_1(x)}{r_0(x)}=1$, $\lim\limits_{x\rightarrow b} \frac{p_1(x)}{p_0(x)}=1$, $\lim\limits_{x\rightarrow b} \frac{q_1(x)-q_0(x)}{r_0(x)}=0$;
	\item[($\beta$)] $q_0/r_0$ is bounded near $b$.
\end{itemize}
Then the following assertions hold:
\begin{itemize}
 \item [{\rm (i)}] $\tau_0$ is limit point at $b$ if and only if $\tau_1$ is limit point  at $b$;
 \item [{\rm (ii)}] $\sigma_\mathrm{ess}(T_0)=\sigma_\mathrm{ess}(T_1)$;
 \item [{\rm (iii)}] $T_0$ and $T_1$ are semibounded from below;
 \item [{\rm (iv)}] $\tau_1-\lambda$ is relatively nonoscillatory with respect to $\tau_0-\lambda$ for every
$\lambda \in \mathbb R \setminus\sigma_{\mathrm{ess}}(T_0)$.
\end{itemize}
\end{theorem}

Observe that by Theorem~\ref{thm:nonoscingap2}~(i) $\tau_0$ is limit circle (or regular) at $b$ if and only if $\tau_1$ is limit circle (or regular) at $b$,
in which case $\sigma_\mathrm{ess}(T_0)=\sigma_\mathrm{ess}(T_1)=\emptyset$.

\begin{remark}
	\label{dontodent}
	Observe that the conditions ($\alpha$) and ($\beta$) in Theorem~\ref{thm:nonoscingap2} are equivalent to the conditions
	\begin{itemize}
	\item[($\alpha'$)] $\lim\limits_{x\rightarrow b} \frac{r_0(x)}{r_1(x)}=1$, $\lim\limits_{x\rightarrow b} \frac{p_0(x)}{p_1(x)}=1$, $\lim\limits_{x\rightarrow b} \frac{q_0(x)-q_1(x)}{r_1(x)}=0$;
	\item[($\beta'$)] $q_1/r_1$ is bounded near $b$.
\end{itemize}
  In fact, this follows immediately from
	\begin{equation*}
		\frac{q_0-q_1}{r_1} = -\frac{q_1-q_0}{r_0}\cdot\frac{r_0}{r_1},\qquad \frac{q_1}{r_1} = \left(\frac{q_1-q_0}{r_0} + \frac{q_0}{r_0}\right)\frac{r_0}{r_1}.
	\end{equation*}
	and hence the roles of $\tau_0$ and $\tau_1$ can be interchanged in the Theorem~\ref{thm:nonoscingap2}.
\end{remark}

\begin{proof}[Proof of Theorem~\ref{thm:nonoscingap2}]
(i) By assumption ($\alpha$) there is $c\in (a,b)$ such that
\begin{equation}
	\label{gme}
	p_0/2 < p_1 < 3p_0/2\quad\text{and}\quad r_0/2 < r_1 < 3r_0/2
\end{equation}
a.\,e.\ on $(c,b)$. This yields $L^2((c,b); r_0) = L^2((c,b); r_1)$.
Since $q_j/r_j$, $j=0,1$, is bounded near $b$ by ($\beta$) and ($\beta'$) (see Remark~\ref{dontodent}), the differential expression $\tau_j$ is in the limit
point case at $b$ if and only if
\begin{equation*}
	 \widehat\tau_j = \frac{1}{r_j} \left(-\frac{\mathrm d}{\mathrm d x}p_j \frac{\mathrm d}{\mathrm d x}\right)
\end{equation*}
is in the limit point case at $b$; cf. \cite[Corollary~7.4.1]{zettl}. Here $\widehat\tau_j u=0$ is explicitly solvable with a fundamental system given by
\begin{equation*}
	u_j(x) = \int_c^x \frac{1}{p_j(t)}\,\mathrm d t,\qquad v_j(x)=1.
\end{equation*}
One has $v_0 = v_1$ and $2/3 u_1\leq u_0\leq 2u_1$ by \eqref{gme}. Hence the number of $L^2$-solutions near $b$ is the same for the differential expressions
$\widehat \tau_0$, $\widehat \tau_1$, $\tau_0$, and $\tau_1$. In particular, this implies (i).
\vskip 0.2cm \noindent
(ii)--(iv)
By condition ($\beta$) there is $d\in (a,b)$ such that
\begin{equation*}
	\lambda_d:=\operatorname*{ess\,inf}_{x\in (d,b)}\frac{q_0(x)}{r_0(x)}>-\infty
\end{equation*}
and, thus, $q_0-\lambda_d r_0\geq 0$ a.\,e.\ on $(d,b)$. This implies that $\tau_0-\lambda_d$ is nonoscillatory (see, e.\,g.\ \cite[Lemma~7.4.1]{zettl}) and hence
$T_0$ is semibounded from below.

Let $\lambda\in\mathbb R\setminus\sigma_{\mathrm{ess}}(T_0)$ and consider the differential expression
\begin{equation*}
	\widetilde \tau_1 = \frac{1}{r_0} \left(-\frac{\mathrm d}{\mathrm d x}p_1 \frac{\mathrm d}{\mathrm d x} + \widetilde q_1\right), \quad \text{where}\quad\widetilde q_1 := q_1 +\lambda r_0 - \lambda r_1,
\end{equation*}
on $(a,b)$. Then $r_0(x)^{-1}(q_0(x)-\widetilde q_1(x)) \rightarrow 0$ as $x\rightarrow b$. Therefore, by \cite[Lemma~4.7]{kt2} applied to $\tau_0-\lambda$ and $\widetilde \tau_1-\lambda$
the differential expression $\widetilde \tau_1-\lambda$ is relatively nonoscillatory with respect to $\tau_0-\lambda$. Because of
\begin{equation*}
	\frac{r_0}{r_1} (\tilde \tau_1-\lambda) u = (\tau_1-\lambda)u
\end{equation*}
the differential equations $(\tau_1-\lambda)u=0$ and $(\widetilde \tau_1-\lambda)u=0$ share the same solutions. This implies that $\tau_1-\lambda$ is relatively nonoscillatory with respect to $\tau_0-\lambda$ for all $\lambda\in\mathbb R\setminus\sigma_{\mathrm{ess}}(T_0)$ and hence
Proposition~\ref{cor:sigess} yields $\sigma_{\mathrm{ess}}(T_1)\subset \sigma_{\mathrm{ess}}(T_0)$.
The same reasoning with the roles of $\tau_0$ and $\tau_1$ reversed together with Remark~\ref{dontodent} shows the  semiboundedness of
$T_1$ and the inclusion $\sigma_{\mathrm{ess}}(T_0)\subset \sigma_{\mathrm{ess}}(T_1)$.
\end{proof}

Note that the relative nonoscillatory property in Theorem~\ref{thm:nonoscingap2}~(iv) does not apply to boundary points of the essential spectrum
and hence no additional information on the possible accumulation of eigenvalues at the boundary of the essential spectrum can be directly obtained.

The following is a straightforward extension of \cite[Theorem~2.1]{kt3} to the case $r_0\neq r_1$. For its formulation suppose that $(\tau_0-\lambda)u=0$ has a positive solution and let $u_0$ be the
corresponding minimal (principal) positive solution of $(\tau_0-\lambda) u_0= 0$ near $b$, that is,
\begin{equation*}
\int_c^b \frac{dt}{p_0(t)u_0(t)^2} = \infty
\end{equation*}
for $c\in(a,b)$.
A second linearly independent solution $v_0$ satisfying $W(u_0,v_0)=1$ is given by d'Alembert's formula,
see, e.g., \cite{BHS20},
\begin{equation}\label{eq:dAl}
v_0(x) := u_0(x) \int_c^x \frac{dt}{p_0(t) u_0(t)^2}.
\end{equation}

\begin{theorem}\label{thm:gu}
Let $\lambda$ denote the minimum of the spectrum of $T_0$
(see Theorem \ref{thm:nonoscingap2}), suppose that $\tau_0-\lambda$ has a positive solution near $b$, and let $u_0$ be a minimal positive solution
near $b$.
Define $v_0$ by d'Alembert's formula \eqref{eq:dAl} and abbreviate
\begin{equation}\label{GritarI}
 \begin{split}
  \Delta(x):=  p_0(x) v_0(x)^2 \bigg(& u_0(x)^2 \big(  q_1(x) -q_0(x) - \lambda( r_1(x) -r_0(x))\big)\\
&+ (p_0(x) u_0'(x))^2 \frac{p_1(x)-p_0(x)}{p_1(x) p_0(x)} \bigg).
 \end{split}
\end{equation}
In addition, suppose
\begin{equation*}
\lim_{x\to b} v_0(x)\, p_0(x) u_0'(x) \frac{p_1(x)-p_0(x)}{p_1(x)} = \lim_{x\to b} \frac{p_1(x)-p_0(x)}{p_1(x)} = 0.
\end{equation*}
Then $\tau_1-\lambda$ is oscillatory if
\begin{equation*}
\limsup_{x\to b} \Delta(x) < -\frac{1}{4}
\end{equation*}
and nonoscillatory if
\begin{equation*}
\liminf_{x\to b} \Delta(x) > -\frac{1}{4}.
\end{equation*}
\end{theorem}

\begin{proof}
This is immediate from \cite[Theorem~2.1]{kt3} since the transformation $q_j \to q_j -\lambda r_j$ reduces everything to the case $\lambda =0$ in which case
$r_j$ becomes irrelevant.
\end{proof}

In the following  we show a variant of Kneser's classical result \cite{Kneser93} (see also \cite[Theorem~9.42 and Corollary~9.43]{tes}).
To this end we recall the iterated logarithm $\log_n(x)$ which is defined recursively via
\[
\log_0(x) := x \quad \text{and}\quad \log_n(x) := \log(\log_{n-1}(x)).
\]
Here we use the convention $\log(x):=\log|x|$ for negative values of $x$. Then
$\log_n(x)$ will be continuous for $x>\mathrm{e}_{n-1}$ and positive for $x>\mathrm{e}_n$, where
$\mathrm{e}_{-1} :=-\infty$ and $\mathrm{e}_n :=\mathrm{e}^{\mathrm{e}_{n-1}}$. Abbreviate further
\[
L_n(x) := \frac{1}{\log_{n+1}'(x)} = \prod_{j=0}^n \log_j(x)\quad\text{and}\quad
Q_n(x) := -\frac{1}{4} \sum_{j=0}^{n-1} \frac{1}{L_j(x)^2}.
\]
Here the usual convention that $\sum_{j=0}^{-1} \equiv 0$ is used, that  is, $Q_0(x)=0$.
In what follows we consider as the underlying interval the interval $(a,\infty)$.

\begin{theorem}\label{blaubaer}
	Consider the Sturm--Liouville differential expression $\tau_1$ on $(a,\infty)$ and assume, in addition, that the limits
	\begin{equation}\label{BadBunny}
		q_\infty:=\lim_{x\rightarrow \infty} q_1(x),\quad p_\infty:=\lim_{x\rightarrow \infty} p_1(x),\quad r_\infty:=\lim_{x\rightarrow \infty} r_1(x)
	\end{equation}
	exist in $\mathbb R$ such that $p_\infty>0$ and $r_\infty>0$. For $n\in\mathbb N_0$ abbreviate
	\begin{align} \label{GritarII}
\widetilde{\Delta}(x):= & L_n(x)^2 \left(\frac{q_1(x)}{p_\infty}- Q_n(x) -\frac{q_\infty}{p_\infty r_\infty}r_1(x)  + \frac{1}{4}
\bigg(\sum_{j=0}^{n-1} \frac{1}{L_j(x)}\bigg)^2
\left(1-\frac{p_\infty}{p_1(x)}\right) \right).
	\end{align}
	Then $\tau_1$ is in the limit-point case at $\infty$,
	every self-adjoint realisation $T_1$ of $\tau_1$ in $L^2((a,\infty);r_1)$ is semibounded from below, and	
	\begin{equation}\label{Farruko}
		\sigma_{\mathrm{ess}}(T_1)=[q_\infty/r_\infty,\infty).
	\end{equation}
	Furthermore, the following assertions hold:
	\begin{enumerate}
		\item If
		      \begin{equation}
			      \label{kneserI}
			      \limsup_{x\rightarrow \infty}{} \widetilde{\Delta}(x) < -\frac{1}{4},
		      \end{equation}
		      then $\sigma(T_1)\cap (-\infty, q_\infty/r_\infty)$ consists of infinitely many simple eigenvalues which accumulate at $q_\infty/r_\infty$;
		\item If
		      \begin{equation}
			      \label{kneserII}
			      \liminf_{x\rightarrow \infty}{} \widetilde{\Delta}(x)^2 > -\frac{1}{4},
		      \end{equation}
		      then $\sigma(T_1)\cap (-\infty, q_\infty/r_\infty)$ consists of finitely many simple eigenvalues.
	\end{enumerate}
\end{theorem}
\begin{proof}
The property of $\tau_1-\lambda$ to be oscillatory or nonoscillatory does not depend
on the left endpoint of the interval $(a,\infty)$. The same applies for the
essential spectrum and the semi-boundedness. Therefore we can assume without
loss of generality that $a=e_n$ and, hence,  $u_0(x):= \sqrt{L_{n-1}(x)}$ is positive,
where we set $L_{-1}(x)=1$.

We choose $r_0(x):=r_\infty$, $p_0(x):=p_\infty$, $q_0(x) := q_\infty + p_\infty Q_n(x)$ and $\lambda:= \frac{q_\infty}{r_\infty}$. One verifies in the same way as in the proof of \cite[Corollary 2.3]{kt3} that $-u_0''+Q_n u_0=0$ and hence
\begin{equation*}
 \left(\tau_0  - \frac{q_\infty}{r_\infty}\right) u_0=0,\quad\text{where}\quad \tau_0=\frac{1}{r_\infty}\left(-\frac{d}{dx}p_\infty\frac{d}{dx}+q_\infty+ p_\infty Q_n\right).
\end{equation*}
It is clear that $u_0$ is the minimal positive solution
near $\infty$ and the solution $v_0$ given by d'Alemberts formula is
$$
v_0(x)= \frac{1}{p_\infty} \sqrt{L_{n-1}(x)} \int_{e_n}^x
\log_{n}'(t)\,dt = \frac{1}{p_\infty} \sqrt{L_{n-1}(x)} \log_{n}(x).
$$
Let $T_0$ be a self-adjoint realization of $\tau_0$ in $L^2((e_n,\infty))$ with
Dirichlet boundary conditions in $e_n$. From $q_0(x)=q_\infty + p_\infty Q_n(x)\geq q_\infty$
for $x\in (e_n,\infty)$ and $\lim_{x\rightarrow \infty} q_0(x)=q_\infty$ we conclude
	\begin{equation*}
		\sigma(T_0)= \sigma_{\mathrm{ess}}(T_0)=[q_\infty/r_\infty,\infty).
	\end{equation*}
By Theorem \ref{thm:nonoscingap2} $\tau_1$ is in limit point at $\infty$, $T_1$
is semibounded and \eqref{Farruko} holds.
For the function $\Delta$ in Theorem~\ref{thm:gu} we obtain
\begin{equation*}
 \begin{split}
  \Delta(x)
&=\frac{1}{p_\infty}\log_n(x)^2 L_{n-1}(x)\Bigg( L_{n-1}(x)\Big(q_1(x)-q_\infty - p_\infty Q_n(x)- \frac{q_\infty}{r_\infty}(r_1(x) -r_\infty)\Big)\\
& \qquad\qquad\qquad\qquad\qquad\quad + \big(p_\infty \sqrt{L_{n-1}(x)}'\big)^2 \frac{p_1(x)-p_\infty}{p_1(x)p_\infty}\Bigg)\\
&=L_{n}^2(x)\Bigg( \frac{q_1(x)}{p_\infty}-  Q_n(x)- \frac{q_\infty}{p_\infty r_\infty}r_1(x) \Bigg)\\
&\qquad\qquad +\log_n(x)^2 L_{n-1}(x)\left(\frac{1}{2\sqrt{L_{n-1}(x)}}L_{n-1}'(x)\right)^2\frac{p_1(x)-p_\infty}{p_1(x)}.
\end{split}
\end{equation*}
We use the formula $L_m'(x)=L_m(x)\sum_{j=0}^m L_j(x)^{-1}$ from \cite{kt3} and conclude
\begin{equation*}
 \Delta(x)  = L_{n}^2(x)\Bigg( \frac{q_1(x)}{p_\infty}-  Q_n(x)- \frac{q_\infty}{p_\infty r_\infty}r_1(x) \Bigg)+L_{n}^2(x)\Bigg( \frac{1}{2} \sum_{j=0}^{n-1}\frac{1}{L_j(x)}\Bigg)^2\frac{p_1(x)-p_\infty}{p_1(x)}.
\end{equation*}
Thus the function $\Delta$ in Theorem~\ref{thm:gu}
coincides with $\widetilde\Delta$.
Now the statements (i) and (ii) follow from Theorem \ref{thm:gu}
and \eqref{zuj}.
\end{proof}

For the special case $n=0$ Theorem~\ref{blaubaer} reduces to the following statement,
which extends the classical Kneser result from \cite{Kneser93} to the case
of non-constant coefficients $p_1$ and $r_1$.

\begin{corollary}\label{kneserli}
Assume that the limits in \eqref{BadBunny} exist in $\mathbb R$ such that $p_\infty>0$ and $r_\infty>0$. Then the following assertions hold:
	\begin{enumerate}
		\item If
		      \begin{equation*}
			      \limsup_{x\rightarrow \infty}{}
			      x^2\Big( \frac{q_1(x)}{p_\infty}-   \frac{q_\infty}{p_\infty r_\infty}r_1(x)\Big)
			      < -\frac{1}{4},
		      \end{equation*}
		      then $\sigma(T_1)\cap (-\infty, q_\infty/r_\infty)$ consists of infinitely many simple eigenvalues which accumulate at $q_\infty/r_\infty$;
		\item If
		      \begin{equation*}
			      \liminf_{x\rightarrow \infty}{}
			      x^2\Big( \frac{q_1(x)}{p_\infty}-   \frac{q_\infty}{p_\infty r_\infty}r_1(x)\Big)
			      > -\frac{1}{4},
		      \end{equation*}
		      then $\sigma(T_1)\cap (-\infty, q_\infty/r_\infty)$ consists of finitely many simple eigenvalues.
	\end{enumerate}
	\end{corollary}

In the next corollary we impose an additional condition on the coefficient $p_1$ and obtain
from Theorem~\ref{blaubaer} for $n\geq 1$
simplified criteria for the
spectrum in $(-\infty, q_\infty/r_\infty)$ to be infinite or finite.

\begin{corollary}\label{cor333}
Assume that the limits in \eqref{BadBunny} exist in $\mathbb R$ such that $p_\infty>0$ and $r_\infty>0$, and let
	\begin{equation}\label{Aventura}
		p_1(x) = p_\infty +o\Big(\frac{x^2}{L_n(x)^2}\Big)
	\end{equation}
for some $n\in \mathbb N$.
	 Then the following assertions hold:
	\begin{enumerate}
		\item If
		      \begin{equation*}
			      \limsup_{x\rightarrow \infty}{}
			      L_{n}^2(x)\Bigg( \frac{q_1(x)}{p_\infty}-  Q_n(x)- \frac{q_\infty}{p_\infty r_\infty}r_1(x)\Bigg)
			      < -\frac{1}{4},
		      \end{equation*}
		      then $\sigma(T_1)\cap (-\infty, q_\infty/r_\infty)$ consists of infinitely many simple eigenvalues which accumulate at $q_\infty/r_\infty$;
		\item If
		      \begin{equation*}
			      \liminf_{x\rightarrow \infty}{}
			      L_{n}^2(x)\Bigg( \frac{q_1(x)}{p_\infty}-  Q_n(x)- \frac{q_\infty}{p_\infty r_\infty}r_1(x)\Bigg)
			      > -\frac{1}{4},
		      \end{equation*}
		      then $\sigma(T_1)\cap (-\infty, q_\infty/r_\infty)$ consists of finitely many simple eigenvalues.
	\end{enumerate}
	\end{corollary}
	\begin{proof}
	Assertions (i) and (ii) follow from Theorem~\ref{blaubaer} if we show that
	\begin{equation}\label{jabitte}
	\lim_{x\rightarrow\infty}L_{n}^2(x)\frac{1}{4}\bigg(\sum_{j=0}^{n-1}\frac{1}{L_j(x)}\bigg)^2\Big(1-\frac{p_\infty}{p_1(x)}\Big)=0.
	\end{equation}
	In fact, it is easy to see that
\begin{equation*}
 \sum_{j=0}^{n-1}\frac{1}{L_j(x)}=\frac{1}{x}+ o(1/x),\quad\text{and hence}\quad \bigg(\sum_{j=0}^{n-1}\frac{1}{L_j(x)}\bigg)^2=\frac{1}{x^2}+ o(1/x^2),
\end{equation*}
that is,
\begin{equation*}
 \Bigg(\sum_{j=0}^{n-1}L_j^{-1}\Bigg)^2=\frac{1}{x^2}+ w(x),\quad\text{where}\quad \lim_{x\rightarrow\infty} x^2 w(x) = 0.
\end{equation*}
Furthermore, from \eqref{Aventura} we conclude
\begin{equation*}
 1-\frac{p_\infty}{p_1(x)}=\frac{p_1(x)-p_\infty}{p_1(x)} =\frac{k(x)}{p_\infty+k(x)},\quad\text{where}\quad \lim_{x\rightarrow\infty}\frac{L_n(x)^2 k(x)}{x^2} = 0,
\end{equation*}
and therefore
\begin{equation*}
\lim_{x\rightarrow\infty} \frac{L_{n}^2(x)}{x^2}  \frac{k(x)}{p_\infty+k(x)}=0\quad\text{and}\quad
 \lim_{x\rightarrow\infty} L_{n}^2(x) w(x) \frac{k(x)}{p_\infty+k(x)}=0.
\end{equation*}
This implies \eqref{jabitte} and hence (i) and (ii) follow.
\end{proof}

As a last result in this context we formulate a variant of Theorem~\ref{thm:gu}, where the pointwise limits are replaced by averaged ones;
cf. \cite[Theorem 2.5]{kt3}.
We leave it to the reader to formulate further generalizations of the results in \cite{kt3} to the case $r_0\not= r_1$ by using the transformation
$q_j \to q_j - \lambda r_j$ from the proof of Theorem~\ref{thm:gu}.

\begin{theorem}
Suppose the same assumptions and the same notation as in Theorem~\ref{thm:gu}. Suppose, in addition, that the functions $\Delta$ and
and $\rho := (p_0 u_0 v_0)^{-1}$ are both bounded and $\rho$ satisfies $\rho=o(1)$ and $\frac{1}{\ell} \int_0^\ell \left|\rho(x+t) -\rho(x) \right| dt = o(\rho(x))$.

Then $\tau_1-\lambda$ is oscillatory if
\begin{equation*}
\inf_{\ell >0} \limsup_{x\to b} \int_x^{x+\ell} \Delta(t) dt < -\frac{1}{4}
\end{equation*}
and $\tau_1-\lambda$ is nonoscillatory if
\begin{equation*}
\sup_{\ell >0}  \liminf_{x\to b} \int_x^{x+\ell} \Delta(t) dt > -\frac{1}{4}.
\end{equation*}
\end{theorem}

\end{document}